\documentclass{article}
\usepackage{amsthm}
\usepackage{amsfonts}
\usepackage{amssymb}
\usepackage{amsmath}
\usepackage{nccmath}
\usepackage{mathtools}
\usepackage{mathrsfs}
\usepackage{setspace}
\usepackage{graphicx}
\usepackage{hyperref}
\usepackage[numbers,sort&compress]{natbib}
\usepackage{enumerate}
\usepackage{authblk}

\newtheorem{theorem}{Theorem}[section]
\newtheorem{lemma}[theorem]{Lemma}

\theoremstyle{definition}
\newtheorem{definition}[theorem]{Definition}
\newtheorem{example}[theorem]{Example}

\newtheorem{corollary}[theorem]{Corollary}

\theoremstyle{remark}
\newtheorem{remark}[theorem]{Remark}

\numberwithin{equation}{section}

\usepackage{amsmath,amssymb,graphicx} 
\usepackage[margin=1in]{geometry} 
\usepackage{enumerate}
\usepackage{authblk}
\usepackage{placeins}
\usepackage{array}
\usepackage{academicons}
\usepackage{xcolor}
\usepackage{svg}
\usepackage[utf8]{inputenc}
\usepackage{tikz}
\usepackage{subcaption}
\usepackage{fancyhdr}
\usepackage{pgfplots}
\usepackage{adjustbox}
\usetikzlibrary{automata} 
\usetikzlibrary{positioning} 
\usetikzlibrary{arrows}
\DeclareMathOperator{\Ima}{Im}

\providecommand{\keywords}[1]{\textbf{\textit{Keywords:}} #1}
\providecommand{\subjclass}[1]{\textbf{\textit{MSC2020:}} #1}
\begin{document}

\nocite{*} 

\title{ Direct sum Decomposition of Spaces of Periodic Functions: $$  \mathbb{P}_n =  \bigoplus \limits_{d|n} \ker(\Phi_d(E))   $$}

\author{Hailu Bikila Yadeta \\ email: \href{mailto:haybik@gmail.com}{haybik@gmail.com} }
  \affil{Salale University, College of Natural Sciences, Department of Mathematics\\ Fiche, Oromia, Ethiopia}
\date{\today}
\maketitle

\begin{abstract}
\noindent It was proved that the space $ \mathbb{P}_p $ of all periodic function of fundamental period $ p  $  is a direct sum of the space $ \mathbb{P}_{p/2} $ of all periodic functions of fundamental period $ p/2 $ and the space $ \mathbb{AP}_{p/2} $  of all antiperiodic functions of fundamental antiperiod $ p/2 $. In this paper, we study some connections between periodic functions, cyclotomic polynomials, roots of unity, circulant matrices, and some classes of difference equations. In particular, we  state and prove the sufficient condition for the existence of periodic solutions of integer period or arbitrary period of some difference equation. We also show that  the space $ \mathbb{P}_n $ of all periodic functions of integer period $n$ can be decomposed as the direct sum of operators' kernels $\ker (\Phi_d(E)) $, where $\Phi_d(E),\, 1 \leq d \leq n, d|n $  are the cyclotomic polynomials of the shift operator $E$. We state and prove important theorems, state and prove the necessary and sufficient conditions for a linear difference equation with constant coefficients to have periodic solutions.
\end{abstract}

\noindent\keywords{periodic function, antiperiodic function, direct sum, decomposition, difference equation, cyclotomic polynomial, root of unity, circulant matrix, kernel }\\
\subjclass{Primary  11R60, 11C20, 39A06, 39A23 }\\
\subjclass{Secondary  47B38, 47B39, 47B92}

\section{Introduction and preliminaries}
\subsection{The shift operator and periodicity }

For $ h \in \mathbb{R} $, we define the shift operator $ E ^h $ and the identity operator $ I $ as
$$ E ^h y(x) := y(x+h),\quad  I y(x) := y(x) .$$
 For $ h = 1 $, we write $ E^h $ only as $ E $ than $ E^{1}$. We agree  that $ E^0 = I $.
We define the forward difference operator $\Delta $ and the back ward difference operators $ \nabla $ as follows
$$ \Delta y(x):=  (E-I)y(x)= y(x+1) -y(x),\quad  \nabla y(x)= (I-E^{-1})y(x)= y(x)-y(x-1). $$
\begin{definition}
  A function $ f$ is said to be $p$-periodic if there exists a $ p >0 $ such that $ f(x)= f(x+ p),\,  x \in \mathbb{R}$. The least such $p$ is called the \emph{period} of $ f $. In terms of shift operator we write this as
  $$ E^pf(x) = f(x) .$$
\end{definition}

\begin{definition}\cite{GN}, \cite{JM}
  A function $ f$ is said to be $p$-antiperiodic if there exists a $ p >0 $ such that $ f(x+ p)= -f(x) ,\,  x \in \mathbb{R}$. The least such $p$ is called the \emph{ antiperiod} of $ f $. In terms of shift operator we write this as
  $$ E^pf(x) = -f(x). $$
  \end{definition}

\begin{example}
  The functions $ f_n(x)= \cos 2n \pi x , \, n\in \mathbb{N} $ are 1-periodic. The functions $g_n(x)= \cos (2n+1) \pi x ,\, n \in \mathbb{N} $ are 1-antiperiodic. The function $f(x)=x-\lfloor x \rfloor$, where $\lfloor x \rfloor$ denotes the greatest integer not greater than $x$, is a 1-periodic function.
\end{example}

\begin{remark}
  Every $p$-antiperiodic function is $2p$-periodic. However not every $2p$-periodic functions is $p$-antiperiodic function. Further properties of $p$-antiperiodic function are available in literatures. For example, finite linear combinations, or  convergent infinite series each of  whose terms are p-periodic (p-antiperiodic) function is a p-periodic(p-antiperiodic) function. For example
  $$ f(x)=\sum_{n =1}^{\infty}\frac{\cos (2n+1) x }{n^2}$$ is  $\pi $-antiperiodic function defined by a uniformly convergent  series each of its terms is $\pi $-antiperiodic. See \cite{GN}.
\end{remark}
Let $ \mathbb{P}_p $ denote the space of all real valued  periodic functions with period equal to $p$
\begin{equation}\label{eq:pperiodic}
  \mathbb{P}_p := \{ f: \mathbb{R}\rightarrow \mathbb{R},\quad f(x+p) =f(x) \}.
\end{equation}
Let $ \mathbb{AP}_p $ denote the space of all real valued  antiperiodic functions with antiperiod equal to $p$.
\begin{equation}\label{eq:pantiperiodic}
  \mathbb{AP}_p := \{ f: \mathbb{R}\rightarrow \mathbb{R },\quad f(x+p) = - f(x) \}.
\end{equation}
\begin{remark}
  The constant function $ f(x)=0 $ is the only function that is both periodic and antiperiodic with any period and antiperiod.
\end{remark}
Let $\mathcal{F} $ represent the space of all real-valued function of real domain. That is
\begin{equation}\label{eq:realvaluedfunction}
\mathcal{F}:= \{ f: \mathbb{R}\rightarrow \mathbb{R }\}.
\end{equation}
The spaces $ \mathbb{P}_p $, and $ \mathbb{AP}_p $  form subspaces of $ \mathcal{F }$. In the authors previous work, a periodic function of period $p$ can be written as  the sum a periodic function period $p/2$ and an antiperiodic function of antiperiod $p/2$ in a unique way. In fact if $f$ is a periodic function of period $p$, then
$$  f = g+h, $$
where
 $$ g= \frac{E^{p/2}f +f}{2},\quad   g= \frac{E^{p/2}f -f}{2}  .  $$

Hence we can perform successive decomposition of the periodic spaces  with the following pattern
\begin{align*}
  \mathbb{P}_p & = \mathbb{AP}_{p/2} \oplus  \mathbb{P}_{p/2} \\
   & =\mathbb{AP}_{p/2}\oplus \mathbb{AP}_{p/4}\oplus \mathbb{P}_{p/4} \\
   & =\mathbb{AP}_{p/2}\oplus \mathbb{AP}_{p/4}\oplus \mathbb{AP}_{p/8} \oplus \mathbb{P}_{p/8}\\
   & =\mathbb{AP}_{p/2}\oplus \mathbb{AP}_{p/4}\oplus \mathbb{AP}_{p/8} \oplus \mathbb{AP}_{p/16}  \oplus \mathbb{P}_{p/16}\\
   & ...........................................................................................\\
   &............................................................................................
\end{align*}
 As such we may decompose a space $\mathbb{P}_p$  of all periodic functions of fundamental period $p$  into an infinite direct sums of spaces of antiperiodic functions of distinct antiperiods. In this paper we consider spaces of periodic functions with some integer period $n$. We consider a finite direct sums of periodic subspaces with periods $d$   that are divisors of $n$. This has some connections to the factorization of the polynomial $x^n-1$ into cyclotomic polynomials.
\subsection{Cyclotomic polynomials}
\begin{definition}
  Let $n \in \mathbb{N} $. An $n$-th cyclotomic polynomial $ \Phi_n$ is an irreducible polynomial with integer coefficients and is a divisor of $x^n-1 $ and is not a divisor of $x^k-1$ for any $k < n$.
  \begin{equation}\label{eq:phin}
  \Phi_n(x)=  \prod_{\substack{1 \leq k \leq n  \\ \text{gcd} (k,n)=1}}(x-e ^{\frac{2 \pi i k }{n}})
\end{equation}
See \cite{TA}, \cite{CMW} \cite{WE}, \cite{WP} \cite{JS}.
\end{definition}
The degree of the $n$th cyclotomic polynomial is $\phi(n) $, where $\phi$ is the \emph{Euler's totient function} which is the count of positive integers that are less than $n$ and that relatively prime to $n$. Here is the list of the first twelve cyclotomic polynomials.
\begin{align*}
   & \Phi_1(x)= x-1 && \Phi_2(x)= x+1 \\
   & \Phi_3(x)= x^2 +x+1 && \Phi_4(x)= x^2 + 1 \\
   & \Phi_5(x)= x^4 + x^3 + x^2 +  x + 1 && \Phi_6(x)= x^2 - x + 1 \\
   & \Phi_7(x)= x^6 + x^5+ x^4 + x^3 + x^2 +  x + 1   && \Phi_8(x)=  x^4  + 1    \\
   & \Phi_9(x)= x^6 + x^3  + 1 &&   \Phi_{10}(x)= x^4 - x^3 + x^2 -  x + 1  \\
   & \Phi_{11}(x)= x^{10}+x^{9}+x^{8}+ x^{7}+ x^6 + x^5+ x^4 + x^3 + x^2 +  x + 1 && \Phi_{12}(x)=  x^4 - x^2  + 1
\end{align*}
A more extended lists of cyclotomic polynomials are available in some literatures. See, for example, \cite{WE}, \cite{WP}.
\subsubsection{The $n$-th roots of unity}
\begin{definition}
  Let $ n \in \mathbb{N} $. An $n$th root of unity is a complex number $x$ satisfying the equation
  $$ x^n = 1 .$$
  An $n$th root of unity is called \emph{primitive }if it is not an $m$th root of unity for some $m$ such that $1\leq m < n $. For example, $e^{i2\pi/3 }$ is primitive $3$rd root of unity.
\end{definition}
An important relation linking cyclotomic polynomials and roots of unity is
\begin{equation}\label{eq:factorizationofxtonminus1}
   x^n -1= \prod_{d|n} \Phi_{d}(x).
\end{equation}
So any  $n$-th root of unity is a is also a root of some cyclotomic polynomial $\Phi_d, d|n$.
\subsection{The shift operators and difference equations}
\begin{definition}
  A linear difference equation of order $n$ is written as
\begin{equation}\label{eq:generallinear}
  p_0(x)y(x+n)+p_1(x)y(x+n-1)+...+p_n(x)y(x)=r(x)
  \end{equation}
where $p_i,\, i=0,1,...,n $ and $r(x)$ are defined on some closed interval $[a,b]$ of the $x$-axis and $ p_0(x)p_n(x) \neq 0 $ on $[a,b]$ .
\end{definition}
In shift operator form
$$  L(x,E)y(x)=r(x),$$
where
$$   L(x,E):= p_0(x) E ^n  + p_1(x)E ^{n-1}+...+p_n(x) I $$
Equation (\ref{eq:generallinear}) is said to be nonhomogeneous if $r(x)\neq 0 $, and homogeneous otherwise. We are interested in linear difference equation with constant coefficients and that are homogenous
\begin{equation}\label{eq:polynomialdifferenceequation}
 P(E)y(x):= a_n y(x+n)+ a_{n-1}y(x+n-1)+...+a_0y(x)=0, \quad a_i \in \mathbb{R},\, i=0,2,3,...,n,\, a_0 a_n \neq 0,
  \end{equation}
where
\begin{equation}\label{eq:polynomialinE}
   P(E)=a_nE^n +a_{n-1}E^{n-1}+...+a_1E+a_0+I
\end{equation}
 is a polynomial function in the shift operator $E$. The shift operator $E$ acts on an exponential function $y(x) = m ^x $ as

\begin{equation}\label{Shiftonexponential}
Ey(x)=E m^x = m^{x+1}=m m^x =my(x).
\end{equation}
The polynomial $P$  in shift operator $E$ acts on an exponential function $y(x) = m ^x $ as
\begin{equation}\label{polynomialofshiftonexponential}
P(E)y(x)=P(E) m^x =P(m) y(x).
\end{equation}

For arbitrary 1-periodic function $\mu $, if we let $y(x)=\mu (x) m^x $, then
    $$Ey(x)=E \mu (x) m^x = \mu (x+1) m^{x+1}= \mu (x) m m^x = my(x).$$

 More generally, for an operator $P(E)$ of a polynomial in shift operator $E$, we have
   $$P(E) \mu(x) r^x = P(r) \mu(x) r^x. $$
   It follows that, for a polynomial $p$ of degree $n$ with distinct roots $m_1, m_2,...,m_n$, the general solution of the difference equation
   $$ P(E)y(x)=0 $$
   is given by
   \begin{equation}\label{eq:solutionform}
     y(x)=\sum_{j=1}^{n} \mu _j(x) m_j^x,
   \end{equation}
 where $  \mu(x)= \lambda $ are arbitrary 1-periodic functions. General theories of linear difference equations, including  the general solution, linear independence and Casoratian determinant etc. are available in different textbooks. For example, see \cite{CHR} \cite{KM} \cite{LB} \cite{MT}. The main purpose of the discussion of difference equations here is to study the connections with periodicity, cyclotomic polynomials, circulant matrices. We also establish a direct sum decomposition of spaces of periodic functions of integer period. consider the difference equation
 \begin{equation}\label{eq:periodicequations}
   y(x+n) - y(x)=0,
 \end{equation}
whose characteristic equation is given by
\begin{equation}\label{eq:charactersticeq}
 \lambda ^n-1=0
\end{equation}
From (\ref{eq:factorizationofxtonminus1})  we get the characteristic roots
 $$ \lambda^n -1= 0 \Leftrightarrow \prod_{d|n} \Phi_{d}(\lambda)=0 \Leftrightarrow \Phi_{d}(\lambda)=0 \quad \text{for some}\quad d|n $$
All the solution of the difference equation  (\ref{eq:periodicequations}) are $n$-periodic.  However some solutions  of the difference equation may have fundamental period that is less than $n$.
 \begin{definition}
   A difference equation $P(E)y(x)=0 $ is termed as \emph{cyclotomic difference equation} if $P$ is a polynomial which is a product of distinct cyclotomic polynomials.
 \end{definition}


 \section{Main Results}
 \subsection{Sufficient condition for existence of periodic solutions, periodic decompositions, circulant matrices, and periodic solutions }

   \begin{theorem}[\textbf{Sufficient condition for existence of a  periodic solution of integer period}]
    Consider the difference equation given in (\ref{eq:polynomialdifferenceequation}).  Let $\lambda \in \mathbb{C} $ be  a root of $P$,  that is also a root of some cyclotomic polynomial. Then the difference equation has a periodic solution of integer period.
   \end{theorem}

   \begin{proof}
     Let $P(\lambda)=0 $, and $ \Phi_n(\lambda)=0 $ for some $n \in \mathbb{N }$. Then by (\ref{polynomialofshiftonexponential}), $P(E) \lambda ^x =P(\lambda  )\lambda ^x =0 $. Therefore $y(x)= \lambda ^x $ is a solution of the difference equation (\ref{eq:polynomialdifferenceequation}). Since $ \Phi_n(\lambda)=0 $, by (\ref{eq:phin}), $\lambda= e ^{\frac{2 \pi k }{n}}$ for some $1 \leq k \leq n $, gcd$(k,n)=1$.
     $$ y(x+n)= \lambda^{x+n}= e ^{\frac{2 \pi k x}{n}} e^{2 \pi k i} = e ^{\frac{2 \pi k x}{n}}=y(x) .     $$
   Therefore,  $y(x)=\lambda ^x = e ^{\frac{2 \pi k x}{n}} $ is a $n$-periodic function.
   \end{proof}

   \begin{theorem}
     Let $P$ be a nonconstant polynomial. Let $y$ be a nontrivial periodic solution, of integer period, of the difference  equation $P(E)y(x)=0 $. Then there exists a $\lambda \in \mathbb{C} $ which is a root of some cyclotomic polynomial $\Phi_n $ and a root  of the polynomial $P$ as well.
   \end{theorem}

   \begin{proof}
   With out loss generality let $P $ has distinct roots, in which case the general solution of the difference equation is of the form
   \begin{equation}\label{eq:generalsolutionform1}
      y(x) = a_r(x) r^x,
   \end{equation}
where the summation is over the distinct roots $r$ of the polynomial $P$ and $ a_r$ are arbitrary $1$-periodic  functions. On the other hand if $y_0$ is  a periodic solution, of integer period say $n$, of the difference equation $P(E)y(x)=0 $. Then
\begin{equation}\label{eq:cyclotomicsolution}
  y_0(x)= \sum_{j=0}^{n-1}b_j(x)\omega_j^x
\end{equation}
for some periodic functions $b_j$ and distinct $n$-th roots of unity $\omega_j,\,j=0,1,...,(n-1)$. Since $y_0$  given in (\ref{eq:cyclotomicsolution}) is derived from the general solution given in (\ref{eq:generalsolutionform1}) by the consideration of linear independence of solutions corresponding to distinct roots, we have for some root $r$ of $P$ and for some root of $ \omega_j $, we have $r= \omega_j $. However $\omega_j $ is a root of some cyclotomic polynomial $\Phi_d $, where $d|n$. This proves the proof of the theorem.
   \end{proof}
\begin{theorem}
  Any solution of a cyclotomic difference equation $P(E)y(x)=0 $ is periodic with some inter period.
\end{theorem}
\begin{remark}
     If $\lambda $ is a root of some cyclotomic polynomial then $ |\lambda|=1 $. However not all $\lambda \in \mathbb{C}, |\lambda|=1  $ is a root of some cyclotomic polynomial. In the next theorem we will see the existence of periodic solution of arbitrary period.
   \end{remark}

   \begin{theorem}[\textbf{Sufficient condition for existence of a  periodic solution of arbitrary period}]
     For some polynomial $P$, the sufficient condition for the difference equation $P(E)y(x)=0$ to have a periodic solution of arbitrary period is that there exists a root $\lambda \in \mathbb{C} $ of $P$ such that $ |\lambda|=1 $.
   \end{theorem}
   \begin{proof}
     Let $\lambda =e ^{i\theta}$, where $\theta$ is the argument of $\lambda $. Then $ y(x)= \lambda^x= e^{i\theta x} $ is a periodic function of period $ \frac{2 \pi }{\theta} $.
   \end{proof}

\begin{definition}
  Let $M$ a linear map from a vector space $V$ to a vector space $W$. That is written as
  $  M : V \rightarrow W $. Then the kernel $\ker M $ and the image $\Ima M $ of the linear operator are defined as
  $$  \ker M =\{x \in V |\,  Mx =0\}, \quad \Ima M= \{y\in W|\, y=Mx,\,\text{ for some}\, x\in V  \} .$$
\end{definition}

\begin{definition}
Let $L$ and $ M $ be two operators on a vector space $X$. We say that $L$ and $M$  have no nontrivial common factor if  whenever $L$ and $M$ can be factored as $L=L' K $ and $M = M'K$  with common factor $K$, then $ K = \alpha I $, where $\alpha \in \mathbb{R}, \alpha \neq 0 $, and $I$ is the identity operator on $X$. We say that $K$ is the greatest common factor $K=\text{gcd}(L,M)$ if $L'$ and $ M'$ have no nontrivial common factor.
\end{definition}

 \begin{lemma}\label{eq:kernelDecompostion}
 Let $L$ and $ M $ be two operators on a vector space $X$, and that commute with each other and that have no common nontrivial factor. Then we have the following result
  $$\ker (LM)= \ker L + \ker M . $$
  \end{lemma}

%

 \begin{proof}
   Let $ f \in \ker (LM) $. Then
   \begin{align*}
     \ker L M & =\{ x \in X \,|\quad LMx =0 \}\\
       & =\{ x \in X\,| \quad  Mx \in \ker L \}\\
       & = \ker M  \cup  ( \Ima M \cap  \ker L )\\
      &  \subset \ker M \cup \ker L \\
      &  \subset \ker M + \ker L
     \end{align*}
   By commutativity of $L$ and $M $, we have
   $$  \ker L M  \subset \ker L + \ker M  $$
   Let $ f \in \ker L + \ker M $. Then $ f = f_{L} +  f_{M}$, where $ f_{L} \in \ker L, f_{M} \in \ker M $, so that $ L f_{L} =0, M f_{M}=0 $. Now $  ML f = M L f_{L} + LMf_{M}=0 $.
   $$ \ker L + \ker M  \subset  \ker L M .   $$
   Hence the theorem is proved.
 \end{proof}

 \begin{remark}
   In the above theorem we included the assumption that $L$ and  $M $ have no nontrivial common factor. Now consider the case $L=M $ so that it is obvious that
   $$  \ker L \subset \ker L^2 . $$
   In this case there may exist an element $ x \in \ker L^2,\,  x \notin \ker L $. So that
   $$\ker L + \ker L = \ker L \subsetneq \ker L^2 $$
   However the operator  $L$ may be idempotent, $ L^2=L $. In this case $\ker L + \ker L = \ker L =\ker L^2 $. In the case of two commuting operators with nontrivial factors the kernel of the product may not be the kernel of the kernels of the individual factors. See the next example.
 \end{remark}
 \begin{example}
   Consider the following linear second order differential equation
   $$ (D-I)^2 y(x)=y''(x)-2y'(x)+y(x)=0. $$
   Then the general solution of the second order homogenous linear differential equation is
   $$ \ker   (D-I)^2 = \{  c_1 e^x + c_2xe^x,\, c_1,c_2 \in \mathbb{R} \}.   $$
   Where as,
   $$     \ker (D-I)+ \ker (D-I)= \ker (D-I) = \{  c e^x |\quad c \in \mathbb{R} \} .$$
 \end{example}

 \begin{example}
   Consider the one-dimensional wave equation
   \begin{equation}\label{eq:waveequation}
    u_{tt}-c^2 u_{xx}=0 .
   \end{equation}
 The general solution to  equation (\ref{eq:waveequation}) is the kernel of the wave operator $\Box:= \partial^2_t- c^2 \partial^2_x $. The operator can be factored as
   $$ \Box =  \partial^2_t- c^2 \partial^2_x  = (\partial_t - c \partial_x) (\partial_t + c \partial_x)= L_1L_2, $$
   where $ L_1:=\partial_t - c \partial_x $, and  $L_2:=\partial_t + c \partial_x $ are operators that commute on appropriate function space for example, a space  $C^2$ of twice continuously differentiable functions. For necessary and sufficient condition for equality of mixed derivative see literatures, for example, \cite{TA1}.
   The general solution $L_1 u=0$, is $ \ker L_1= \{   f(x+ct)| \quad  f \in C^2 \}$, and the general solution $L_2 u=0 $ is $ \ker L_2= \{   g(x-ct)| \quad g \in C^2 \}$. Consequently, according to Theorem \ref{eq:kernelDecompostion}, we have the general solution of the wave equation (\ref{eq:waveequation})
   $$ u(x,t)= f(x+ct)+g(x-ct), \quad f,g \in C^2. $$
   \end{example}

   \begin{theorem}\label{eq:kerneldecomposition}
   Let $\mathbb{P}_n$ be the space of all periodic function with period equal to $n\in N $, and $E$ is the shift operators' kernels. Then we have the following decomposition of spaces of periodic functions into subspaces as a direct sum
     $$ \mathbb{P}_n =  \bigoplus \limits_{d|n} \ker(\Phi_d(E)). $$
   \end{theorem}
\begin{proof}
  The proof follows from (\ref{eq:factorizationofxtonminus1}) and  Lemma \ref{eq:kernelDecompostion} inductively, as $\Phi_d(E)$  and $ \Phi_{d'}(E) $  have no common nontrivial factor for  $d'\neq d,\,    d',d | n $. So that we have
  $$ \mathbb{P}_n = \sum \limits_{d|n} \ker(\Phi_d(E)) $$
The remaining property to prove is that
$$\ker(\Phi_d(E))\cap \ker(\Phi_{d'}(E)) =\{ 0 \}, d,d'|n, d \neq d' . $$
This follows from the distinctness of the roots of the polynomial $x^n-1$ and its factors $\Phi_{d}(x) $ and $\Phi_{d'}(x) $  have no common roots. This guarantees that the intersections of the kernels of the operators $\Phi_d(E)$ and $\Phi_{d'}(E)$ is  $ \{ 0 \}$.
\end{proof}

\begin{corollary}
   Consider the linear $n$-th order difference equation $y(x+n)-y(x)=(E^n-I)y(x)=0$. If $y =f(x)$ is the solution of the difference equation, then $f(x)$ is of the form
  $$ f(x)= \sum_{d|n}f_d(x),\, f_d \in \ker(\Phi_d(E)) $$
  \end{corollary}

\begin{example}
  Note that $\ker \Phi_1(E) =\mathbb{P}_1 $, $\ker \Phi_2(E) =\mathbb{AP}_1 $, $\ker \Phi_2(E) =\mathbb{AP}_2 $,  $\ker \Phi_8 (E)=\mathbb{AP}_4 $
\end{example}

\begin{example}
  We have the following decomposition the space $\mathbb{P}_{12} $ of periodic functions of period equal to $12$.
Consider the factorization
\begin{align*}
  x^{12}-1 &= (x-1)(x+1)(x^2+x+1)(x^2+1)(x^2-x+1)(x^4-x^2+1) \\
   & = \Phi_1(x)\Phi_2(x)\Phi_3(x)\Phi_4(x)\Phi_6(x)\Phi_{12}(x)
\end{align*}
Then by Theorem \ref{eq:kerneldecomposition}
  \begin{align*}
  \mathbb{P}_{12} &= \ker (\Phi_1(E))\oplus  \ker(\Phi_2(E)) \oplus  \ker(\Phi_3(E)) \oplus \ker(\Phi_4(E)) \oplus \ker(\Phi_6(E)) \oplus \ker (\Phi_{12}(E)) \\
   & = \mathbb{P}_1 \oplus \mathbb{AP}_1 \oplus \ker(\Phi_3(E)) \oplus  \mathbb{AP}_2 \oplus \ker(\Phi_6(E)) \oplus \ker(\Phi_{12}(E)).
\end{align*}
See \cite{HBY}.
\end{example}

\begin{example}
  Find a linear difference equation of lowest order that has the function $y(x)=\cos \pi x + \sin (2 \pi x/3 )$ as a solution. Let $ f_1(x)= \cos \pi x $, and $ f_2(x)= \sin (2 \pi x/3) $. Then we observe that $f_1\in \ker \Phi_2(E)= \mathbb{AP}_2 $,  $f_2 \in \ker \Phi_3(E) $. Since $\Phi_2(E) $ and $\Phi_3(E) $ have no common factor, the required difference equation is
  $$ \Phi_2(E)\Phi_3(E)y(x)= (E^3+2E^2+2E+I)y(x)=0  $$
\end{example}

\subsection{Circulant Matrices and difference Equations}
\begin{definition}
  A circulant matrix $C(a_0,a_1,...,a_{n -1})$ is an  $n \times n $  matrix of the form
\begin{equation}\label{eq:circulantsystem}
 \begin{bmatrix}
   a_{0} & a_{1} & \cdots & a_{n-1} \\
   a_{n-1} & a_{0} & \cdots & a_{n-2} \\
   \vdots  & \vdots  & \ddots & \vdots  \\
   a_{1} & a_{2} & \cdots & a_{0}
 \end{bmatrix},
\end{equation}
where $a_0,a_1,...,a_{n-1} \in \mathbb{Q} $. It is said to be unital if $a_0,a_1,...,a_{n-1} \in \{0,1\} $. See \cite{ZC}.
\end{definition}
The determinant of a circulant matrix    $C(a_0,a_1,...,a_{n -1})$ is calculated as
  $$ \det C(a_0,a_1,...,a_{n -1})=\prod_{j=0}^{n-1}( a_0+a_1\omega_j + ...+a_{n-1}\omega_j^{n-1}),$$
where
$$   \omega_j = e ^{\frac{2 \pi j}{n}},\quad  i^2 =-1 .$$
See \cite{ZC} and the references cited therein.
\begin{definition}
  The polynomial
  $$ f(x)=  a_0+a_1x + a_2x^2+...+a_{n-1}x^{n-1} \in \mathbb{Q}[x]$$
  is called \emph{associated polynomial} of  $C(a_0,a_1,...,a_{n -1})$. It is said to be unital if $a_0,a_1,...,a_{n-1} \in \{0,1\} $. See \cite{ZC}.
\end{definition}
\begin{theorem}
  Let $P(x)\in \mathbb{Q}[x]$. If $P$ is an associated polynomial of some singular circulant matrix, then the linear difference equation $P(E)y(x)=0 $ has a periodic solution.
\end{theorem}
\begin{proof}
  \begin{equation*}
   \det C(a_0,a_1,...,a_{n -1})=\prod_{j=0}^{n-1}f(\omega_j)
  \end{equation*}
  Then $   \det C(a_0,a_1,...,a_{n -1})=0 $ implies that $f(\omega_j)=0 $ for some $j=0,1,...,n-1$.
\end{proof}

\begin{example}
Let $p(x) = x-1 = \Phi_1(x)$. $ \Phi_1$ is an associated polynomial of a singular circulant matrix, $$ A= \begin{bmatrix}
                                -1 & 1\\
                                1 & -1
                              \end{bmatrix}.$$
\end{example}

\begin{theorem}\label{eq:lineardependence}
  Any periodic function $y$ of  fundamental  period $n \in \mathbb{N} $ is a solution of some difference equation
  $$ (a_0I+a_1E+ a_2E^2+...+a_{n-1}E^{n-1})y(x)=0  $$
\end{theorem}

\begin{proof}
  Since $y \in \mathbb{P}_n $, we have $E^n y =y $. For any $ m \in\mathbb{ N} $, by division algorithm, $m=nq+r, 0 \leq r < n $ for some $ q, $, so that
  $$ E^my= E^{nq+r}y = E^r (E^{nq} y)=E^ry.    $$
  Therefore the distinct elements are
  $$ y=Iy,\quad Ey,\quad E^2y,\quad ...\quad ,E^{n-1}y.  $$
  Due to linear dependence we have
  $$a_0y + a_1E y + a_2E^2y +...+a_{n-1}E^{n-1}y=0,   $$
  for appropriate selections of the coefficients $a_0, a_1,...,a_{n-1} $ not all zero.
\end{proof}
\begin{theorem}
  Let $y$ be a periodic function of integer fundamental period $n \in \mathbb{N} $. The the coefficients $a_0, a_1,...,a_{n-1} $ in Theorem \ref{eq:lineardependence} are determined by the homogeneous system of equation
  $$ M x=0 ,$$
  where $ M $ is the $n \times n $ circulate matrix given in (\ref{eq:circulantsystem}) and $x$ is the column vector of unknowns $x=(a_0,\, a_1,\, ...,\,a_{n-1})^T $.
  \begin{proof}
    Let $y \in \mathbb{P}_n  $. By Theorem (\ref{eq:lineardependence}), we have
    \begin{equation}\label{eq:firstequation}
     a_0y + a_1E y + a_2E^2y +...+a_{n-1}E^{n-1}y =0,
    \end{equation}
  Now applying the shift operator to the equation in (\ref{eq:firstequation})and tanking into periodic of $y$ of period $n$ we get
  \begin{equation}\label{eq:secondequation}
    a_{n-1}y+a_1E^2y+...+a_{n-2}E^{n-1}y=0.
  \end{equation}
  Applying  the shift operators $E, E^2, E^3, ...,E^{n-1} $ to (\ref{eq:firstequation}) and taking the linear equations thus formed as well as the original equation (\ref{eq:firstequation}), we get the desired result.
  \end{proof}
\end{theorem}
\section{Conclusions}
In this paper, we have discussed some connections with shift operators, periodicity, difference equations, cyclotomic polynomials and roots of unity. The main results of the paper includes the result on the  kernels of a product of two operators with no nontrivial factors, the decomposition of space $\mathbb{P}_n$ of periodic functions of integer period $n$ into the direct some of the kernels of the $\Phi_d(E)$ the cyclotomic polynomials in the shift operator $E$. It is not that not all linear difference equations have periodic solutions. The paper included the  sufficient condition that a difference equation can have a periodic solution of integer period, and arbitrary period. Difference equations with integer period $n$ have some connections with  $n$th roots of unity and cyclotomic polynomials $\Phi_d,\, d|n $. Difference equations are connected with in some way with periodic function that  even solution of the equations  with no periodic solutions are linear combinations of some functions over arbitrary periodic functions. Circulant matrices are also important in the study of difference equations .

\section*{Conflict of interests}
The author declare that there is no conflict of interests regarding the publication of this paper.

\section*{Acknowledgment}
The author is thankful to the anonymous reviewers for their constructive and valuable suggestions.

\section*{Data availability}
The are no external data used in this paper other that the reference materials used here.

\end{document}